\documentclass[11pt]{amsart}  
  
\usepackage{amssymb}
\usepackage{amsmath}
\usepackage{amsfonts}
\usepackage[usenames]{color}
\usepackage{graphicx}
\usepackage{array}

\input{epsf}
\usepackage{graphicx}
\usepackage{subfig}
\usepackage{psfrag}
\usepackage{epsfig}

\makeatletter
 
 \@addtoreset{equation}{section}
\makeatother

\newtheorem{theorem}{Theorem}[section]
\newtheorem{lemma}[theorem]{Lemma}
\newtheorem{proposition}[theorem]{Proposition}

\newtheorem{corollary}[theorem]{Corollary}

\theoremstyle{definition}
\newtheorem{definition}[theorem]{Definition}
\newtheorem{assumption}[theorem]{Assumption}

\theoremstyle{remark}
\newtheorem{remark}[theorem]{Remark}

\numberwithin{equation}{section}

\newcommand{\Z}{\mathbb{Z}}

\newcommand{\A}{\mathfrak{A}}
\begin{document}

\title{The Self-Linking Number in Planar Open Book Decompositions}

\author{Keiko Kawamuro}
\address{Department of Mathematics \\ 
The University of Iowa \\ Iowa City, IA 52240}
\email{kawamuro@iowa.uiowa.edu}


\date{\today}


\begin{abstract}

We construct a Seifert surface for a given null-homologous transverse link in a contact manifold that is compatible with a planar open book decomposition, then obtain a formula of the self-linking number.
It extends Bennequin's self-linking formula for braids in the standard contact $3$-sphere.

\end{abstract}

\maketitle



\section{Statement of the main theorem}

Let $S=S_{0, r}$ be an oriented $S^2$ with $r$ disks removed. 
See Figure~\ref{setting}. 
\begin{figure}[htpb!]
\begin{center}
\psfrag{G}{$\gamma$}
\psfrag{d}{$d$}
\psfrag{R}{$\rho$}
\psfrag{x}{$x$}
\psfrag{s}{$\sigma$}
\psfrag{e}{$e_{2,4}$}
\psfrag{A}{$\alpha_{2,4}$}
\psfrag{a}{$\alpha$}
\psfrag{1}{$\scriptstyle1$}
\psfrag{2}{$\scriptstyle 2$}
\psfrag{3}{$\scriptstyle 3$}
\psfrag{4}{$\scriptstyle 4$}
\psfrag{r}{$\scriptstyle r$}
\psfrag{n}{$\scriptstyle n$}
\includegraphics[width=.7\textwidth]{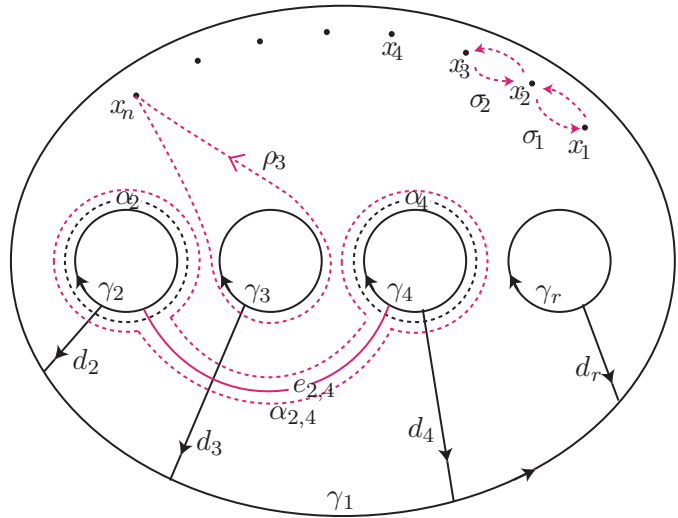}
\end{center}
\caption{Surface $S$.}
\label{setting}
\end{figure}
We call the boundary circles $\gamma_1, \cdots, \gamma_r$, whose orientations are induced by that of $S$.
Let $\alpha_i \subset S$ $(i=1, \cdots, r)$ be a circle parallel to $\gamma_i$ and disjoint form other $\alpha_j$. 
Let $A_i \in {\rm Mod}(S)$ denotes the positive Dehn twist about $\alpha_i$. 
See Figure~\ref{rh-twist}.
\begin{figure}[htpb!]
\begin{center}
\psfrag{A}{$\alpha$}
\includegraphics[width=.5\textwidth]{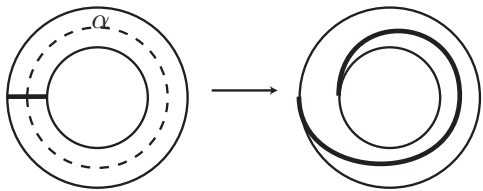}
\end{center}
\caption{A positive Dehn twist $A_\alpha$ about $\alpha$. }\label{rh-twist}
\end{figure}
Let $d_i$ ($i=2, \cdots, r$) be an arc from a point on $\gamma_i$ to a point on $\gamma_1$ with $d_i$ and $d_j$ disjoint.
For $2\leq i < j \leq r$, let $e_{i, j}$ be an arc connecting $\gamma_i$ and $\gamma_j$ such that it intersects $d_{i+1}, \cdots, d_{j-1},$ once for each of them in this order. 
Let $\alpha_{i, j}$ be a loop surrounding $\gamma_i \cup e_{i,j} \cup \gamma_j$, and $A_{i,j} \in {\rm Mod}(S)$ be the positive Dehn twist about $\alpha_{i,j}$. 
One can see that the $A_{i, j}$'s are the standard generators of the pure braid group of $(r-1)$ strands, see \cite[Lemma 1.8.2]{B}.

Since loops $\alpha_h$ and $\alpha_{i,j}$ do not intersect, the Dehn twists $A_h$ and $A_{i, j}$ commute. 
Viewing $A_1$ as a full twist of an $(r-1)$-stranded braid, $A_1$ can be  written as a product of $A_{i,j}$'s.
Let $\phi: S \to S$ be a diffeomorphism fixing the boundary $\partial S$ point-wise.
It is represented, up to isotopy, as a product of $A_i$ and $A_{i,j}$:
\begin{equation}\label{eq of phi}
\phi =  {(A_{i_l, j_l})}^{\epsilon_l} \cdots (A_{i_1, j_1})^{\epsilon_1} (A_r)^{k_r} \cdots (A_2)^{k_2}, 
\qquad \epsilon_i \in \Z \setminus\{0\},\  k_i  \in \Z,
\end{equation}
read from the right to left.
Let $M=M_{(S, \phi)}$ be the closed $3$-manifold admitting the open book decomposition $(S, \phi)$.
Namely, 
$
M_{(S, \phi)} = S \times [0,1] / \sim
$
where ``$\sim$'' is an equivalence relation identifying $(\phi(x), 0)$ and  $(x,1)$ for $x\in {\rm Int}(S)$, and identifying $(x, \tau)$ and $(x, 1)$ for $x \in \partial S$, $\tau \in [0, 1]$.

By Giroux's seminal work \cite{G}, there is a one to one correspondence between contact $3$-manifolds up to contact isotopy and open book decompositions up to positive stabilizations. 
By $\xi_{(S, \phi)}$, we denote the contact structure on $M_{(S, \phi)}$ corresponding to $(S, \phi)$.

\begin{remark}
In \cite[Theorem 1]{E1}, Etnyre shows that any overtwisted contact structure is supported by a planar open book decomposition. 
\end{remark}

A positively oriented {\em transverse} knot $K$ in a contact $3$-manifold $(M, \xi)$ is an embedding of $S^1$ such that at each point $p \in K$, the contact plane $\xi_p$ and $K$ have positive transverse intersection.
We say that transverse knots $K_0$ and $K_1$ are {\em transversely isotopic}, if there exists a smooth $1$-parameter family $K_t$ ($0\leq t \leq 1$) of transverse knots in the ambient contact manifold.

The {\em self linking number} is an invariant of null-homologous transverse links. 
For a given null-homologous transverse link $K$ and the relative homology class of a Seifert surface $[\Sigma]\in H_2(M, K; \Z)$, the self linking number $sl(K, [\Sigma])$ is an obstruction extending a nowhere vanishing smooth vector field $\{\vec 0\neq v_p \subset T_p\Sigma \cap \xi_p : p \in K\}$ over $K$ to a nowhere vanishing vector field over $\Sigma$ with values in $\xi_p$. See \cite[Def 3.5.28]{Ge} for precise definition.

In this paper we study the self linking number via braid theory.
By a {\em braid} in open book decomposition $(S, \phi)$, we mean a knot (or link) in $M_{(S, \phi)}$ that intersects positively each page of the open book. 
Due to Bennequin \cite{Ben} (for the standard contact $3$-sphere) and Pavelescu \cite{P} (for general case), we can identify a transverse knot in $(M_{(S, \phi)}, \xi_{(S, \phi)})$ up to transverse isotopy and a braid in $(S, \phi)$ up to braid isotopy and positive braid stabilization.

\begin{assumption}\label{assumption}
Let $b$ be an $n$-stranded braid in an open book $(S, \phi)$. 
Applying braid isotopy, we may assume that the braid $b$ intersects the page $S \times \{1\}$ in $n$ points, $x_1, \cdots, x_n$, sitting between $\gamma_1$ and $\alpha_1$ counterclockwise in this order. See Figure~\ref{setting}.
\end{assumption}

Next we define braid words $\sigma_i$ and $\rho_i$, see Figure~\ref{setting}: 
Consider an $n$-stranded trivial braid: $\bigcup_{i=1}^n (\{x_i\} \times [0,1])$ in $S \times [0,1]$.
We denote the local half right twist of the $i$-th and $(i+1)$-th strands by $\sigma_i$. 
Let $\rho_i$ ($i=2, \cdots, r$) denote the winding of the $n$-th strand once around the binding $\gamma_i$ counterclockwise.

\begin{proposition}
Any $n$-stranded braid $b$ in $S \times [0,1]$ is represented by a braid word in $\{ \sigma_1, \cdots, \sigma_{n-1}, \rho_2, \cdots, \rho_r \}$. 
\end{proposition}

\begin{proof}
Let $S^*$ be the surface $S$ with $n$ marked points $x_1, \cdots, x_n$.
Let $C(S, n)$ denote the configuration space of $n$ distinct unordered points in $S$. 
The fundamental group $\pi_1(C(S, n))$ is the $n$-stranded surface braid group of $S$.
Recall the generalized Birman exact sequence \cite{B}, \cite[Theorem 9.1]{FM}.
$$
1 
\longrightarrow \pi_1(C(S, n))  
\stackrel{\mathcal Push}{\longrightarrow}
{\rm Mod}(S^*)  
\stackrel{\mathcal Forget}\longrightarrow 
{\rm Mod}(S)  
\longrightarrow 1
$$
The map ${\mathcal Forget}$ is forgetting the $n$ marked points.
Hence $\ker({\mathcal Forget})=\pi_1(C(S, n))$ is generated by $\{ \sigma_1, \cdots, \sigma_{n-1}, \rho_2, \cdots, \rho_r \}$.
\end{proof}

We fix a braid word for a braid $b$ and define $a_\sigma$ to be the exponent sum of all the $\sigma_i$'s in the braid word.
Also, for each $j=2, \cdots, r$, let $a_{\rho_j}$ denote the exponent sum of $\rho_j$ in the braid word.

By Etnyre-Ozbagci~\cite{EO}, the $1$st homology group of $M_{(S, \phi)}$ is
\begin{equation}\label{EO-eq}
H_1(M_{(S, \phi)}; \Z)=\langle [\gamma_2], \cdots, [\gamma_r] \ | \ [d_i] - \phi_\ast [d_i]=0, \ i=2, \cdots, r \rangle.
\end{equation}
Let $t_{i,j}$ be integers such that 
\begin{equation}\label{rewrite}
[d_i] - \phi_\ast [d_i] = \sum_{j=2}^r t_{i,j} [\gamma_j] 
\quad \mbox{ in } H_1(S;\Z).
\end{equation}
If $b$ is null-homologous in $M$, then there exist integers $s_j$ such that
\begin{equation}\label{def of s_j}
\sum_{j=2}^r a_{\rho_j} [ \gamma_j ] =  \sum_{i=2}^r s_i \sum_{j=2}^r t_{i,j} [\gamma_j] \quad \mbox{ in } H_1(S;\Z).
\end{equation}
Comparing the coefficients, we get
\begin{equation}\label{a-rho-and-T}
a_{\rho_j} = \sum_{i=2}^r s_i \ t_{i,j}.
\end{equation}

Here is our main theorem, which will be proved in Section~\ref{proof-section}.

\begin{theorem}\label{sl-thm}
Let $b$ be a null-homologous $n$-stranded oriented braid in an open book $(S, \phi)$.
Let $a_\sigma, a_{\rho_j}, s_j, t_{i,j}$ be integers as defined above.
There is a 
Seifert surface $\Sigma \subset M$ such that the self-linking number of the braid relative to the class $[\Sigma] \in H_2(M, b;\Z)$ is computed as follows:
\begin{equation}\label{sl-formula}
sl(b, [\Sigma]) = 
-n + a_\sigma + \sum_{j=2}^r a_{\rho_j} (1 - s_j) - \sum_{j=2}^r s_j \sum_{{2\leq i \leq r} \atop{i \neq j}} t_{j,i}
\end{equation}
\end{theorem}

\begin{remark}\label{remark by Matt}
For a generic $M_{(S, \phi)}$ with $S$ planar, 
the formula (\ref{sl-formula}) is independent of choice of homology classes by the following reason:
Suppose $[\Sigma], [\Sigma'] \in H_2(M, b; \Z)$. 
We have
$sl(b, [\Sigma]) - sl(b, [\Sigma']) = -\langle
e(\xi), [\Sigma] - [\Sigma'] \rangle,$
where $e(\xi) \in H^2(M;\Z)$ is the Euler class for the $2$-plane bundle $\xi$ and $[\Sigma] - [\Sigma'] \in H_2(M;\Z)$ (see Proposition 3.5.30 of \cite{Ge} for example). 
By Poincar\'e duality along with (\ref{EO-eq}), (\ref{rewrite}), if the matrix $\{ t_{i,j}\}$ has the full rank, ${\rm rk}\{ t_{i,j}\}=r-1$, then the cohomology group $H^2(M_{(S, \phi)};\Z)$ is a torsion group. 
Hence $\langle e(\xi), [\Sigma] - [\Sigma'] \rangle=0.$ 
\end{remark}

\begin{remark}
Theorem~\ref{sl-thm} is a generalization of Bennequin's work \cite{Ben}, where he studied the case when $S=S_{0,1}=D^2$, $\phi= id$ and obtained that 
$sl(b) = -n + a_\sigma$.
Recently, Pavelescu and the author investigated the cases when $S$ is an annulus or a pair of pants in \cite{KP}. 
Our formula (\ref{sl-formula}) also extends their result. 
\end{remark}

To state a corollary, we recall the stabilization of a braid in an open book decomposition. 
\begin{definition}\label{def of stabilization} 
Let $b$ be a closed braid in an open book $(S, \phi)$ (only in this definition, we do not require planar property of $S$).
Suppose that $\lambda$ is a binding component of the open book and $p\in ((S \times \{\theta\}) \cap b)$ is a point, see Figure~\ref{stabilization-def}.
Join $p$ and a point on $\lambda$ by an arc $a \subset ((S \times \{\theta\}) \setminus b)$, which is allowed to have self-intersections.
A {\em positive $($negative$)$ stabilization of $b$ about $\lambda$ along $a$} is pulling a small braid segment containing $p$ along $a$, then adding a positive (negative) kink about $\lambda$ at the end of $a$.
\end{definition}
\begin{figure}[htpb!]
\begin{center}
\begin{picture}(300, 85)
\put(0,0){\includegraphics{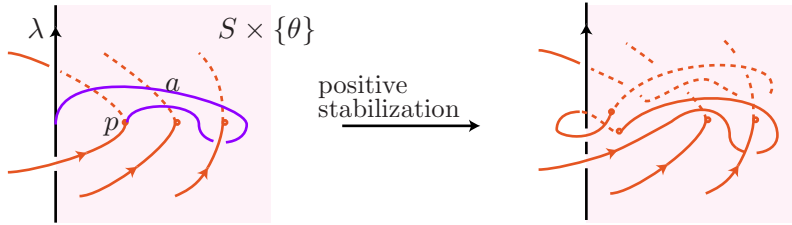}}
\put(118, 49){\small positive}
\put(118, 40){\small stabilization}
\put(60, 50){$a$}
\put(37, 35){$p$}
\put(80, 70){$S \times \{\theta\}$}
\put(8, 70){$\lambda$}
\end{picture}
\caption{The positive braid stabilization along $a$.}\label{stabilization-def}
\end{center}
\end{figure}

The following is proved in Section~\ref{proof-section}.
\begin{corollary}\label{cor-neg-stab}
A braid in $(S, \phi)$ and its negative braid stabilization are transversely  non-isotopic.
More precisely, their self linking numbers differ by $2$. 
\end{corollary}

As for a positive braid stabilization, it preserves the transverse knot type of any braid (see \cite{P}, for example).
In fact, a similar calculation in the proof of Corollary~\ref{cor-neg-stab} verifies that  our self-linking formula (\ref{sl-formula}) is invariant under a positive braid stabilization.

\section{Bennequin surface and $\A$-annulus}

Given a braid word for $b$ in $\{\sigma_1, \cdots, \sigma_{n-1}, \rho_2, \cdots, \rho_r\}$, we construct a Seifert surface $\Sigma$ step by step. 
Our construction is a generalization of that in \cite{KP}, where annulus and pants open book decompositions are studied.

First, consider $n$ copies of disks $\delta_1, \cdots, \delta_n$, whose centers are pierced by $\gamma_1$ and the boundary $\partial \delta_i$ is the unknot braid $(\{x_i \} \times [0, 1]/\sim) \subset M_{(S, \phi)}$. 
See Figure~\ref{D-disk}.
%
\begin{figure}[htpb!]
\begin{center}
\psfrag{G}{$\gamma$}
\psfrag{D}{$\delta$}
\psfrag{s}{$\sigma$}
\psfrag{t}{$\sigma_{i+1}^{-1}$}
\psfrag{x}{$x$}
\psfrag{1}{$\scriptstyle1$}
\psfrag{2}{$\scriptstyle2$}
\psfrag{3}{$\scriptstyle3$}
\psfrag{r}{$\scriptstyle r$}
\psfrag{i}{$\scriptstyle i$}
\psfrag{j}{$\scriptstyle i+1$}
\psfrag{n}{$\scriptstyle n$}
\includegraphics[width=.6\textwidth]{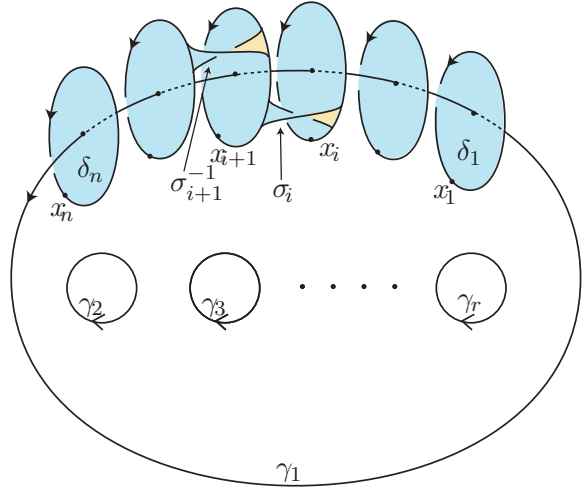}
\end{center}
\caption{$\delta$-disks and twisted bands.}
\label{D-disk}
\end{figure}
Here the fact that $\phi = id$ near $\gamma_1$ is essential.
The orientation of $\delta_i$ is induced by the braid.
The characteristic foliation on $\delta_i$ has one positive elliptic singularity at the center.

Second, for each $\sigma_i$ (resp. $\sigma_i^{-1}$) in the braid word of $b$, we attach a positively (resp. negatively) twisted band between $\delta_i$ and $\delta_{i+1}$.
The characteristic foliation on the twisted band has one positive (resp. negative) hyperbolic singularity.

If the braid word does not contain any $\rho_j$ then this surface is the desired Seifert surface, known as the Bennequin surface~\cite{Ben}.
We call it $\Sigma$.

\subsection{Surface $\tilde \Sigma_b$ }\label{A-annulus-section}
\qquad  \newline
In the following, we assume that the braid word contains $\rho_j$s.
Since careful investigation of the characteristic foliation will be required later, here we orient the leaves of the characteristic foliation, cf. \cite[page 80]{OS}. Let $F$ be an oriented surface in a contact manifold . For $p \in F$, a nonsingular point of a leaf $L$ of the characteristic foliation, let $\vec{n}\in T_p F$ be a positive normal vector to $\xi_p$. 
We choose a tangent vector $\vec{v} \in T_pL$ so that $\{\vec{v}, \vec{n}\}$ is a positive basis of $T_p F$.
This vector field $\vec{v}$ determines the orientation of the leaf $L$.

In Section~\ref{glueing-section} below, we will alter the contact structure in $S \times [0, \varepsilon]$ so that we can glue $S\times\{0\}$ and $S\times \{1\}$ smoothly by the monodromy $\phi$. 
However for the moment, we assume homogeneity of the characteristic foliation:

\begin{assumption}\label{foliation-assumption}
(1) 
Following Thurston-Winkelnkemper's idea \cite{TW}, we may assume that each page $S\times \{t\}$ of the open book has the characteristic foliation as illustrated in Figure~\ref{foliation-of-S}.
The contact planes are positively tangent to the page along the dashed circle.
All the hyperbolic points between the consecutive $\gamma$-circles have positive sign.
The contact planes and bindings intersect positively. 

(2)
The Dehn twist curves $\alpha_{i, j}$ are in the region enclosed by the dashed line.

(3)
By braid isotopy, we may assume that the braid segment for $\rho_j$ projected to $S \times \{1\}$ is mostly inside the dashed circle and does not enclose the hyperbolic points between $\gamma$-circles.
See Figure \ref{foliation-of-S}. 
\end{assumption}

\begin{figure}[htpb!]
\begin{center}
\psfrag{+}{$\rho_2$}
\psfrag{-}{$\rho_3^{-1}$}
\psfrag{1}{$\gamma_1$}
\psfrag{2}{$\gamma_2$}
\psfrag{3}{$\gamma_3$}
\psfrag{r}{$\gamma_r$}
\psfrag{n}{$x_n$}
\includegraphics[width=.6\textwidth]{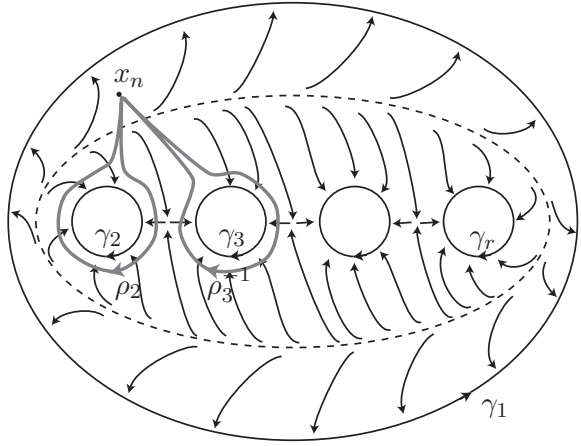}
\end{center}
\caption{The characteristic foliation on $S$.  }
\label{foliation-of-S}
\end{figure}

Now we introduce the $\mathfrak A$-annulus.

\begin{definition}(\textbf{Annulus $\A$})
For each $\rho_j$ (resp. $\rho_j^{-1}$) in the braid word of $b$ we attach an oriented annulus, called an $\mathfrak A$-{\em annulus}, to $\delta_n$.
See Figure~\ref{A-annulus}.
%
\begin{figure}[htpb!]
\begin{center}

\psfrag{d}{$\delta_n$-disk}
\psfrag{g}{$\gamma_j$}
\psfrag{r}{$\rho_j$}
\psfrag{R}{$\rho_j^{-1}$}
\psfrag{b}{braid}
\psfrag{c}{$c$}
\psfrag{A}{$\mathfrak A$-annulus}
\psfrag{+}{$\mathfrak A^+$}
\psfrag{-}{$\mathfrak A^-$}

\includegraphics[width=.9\textwidth]{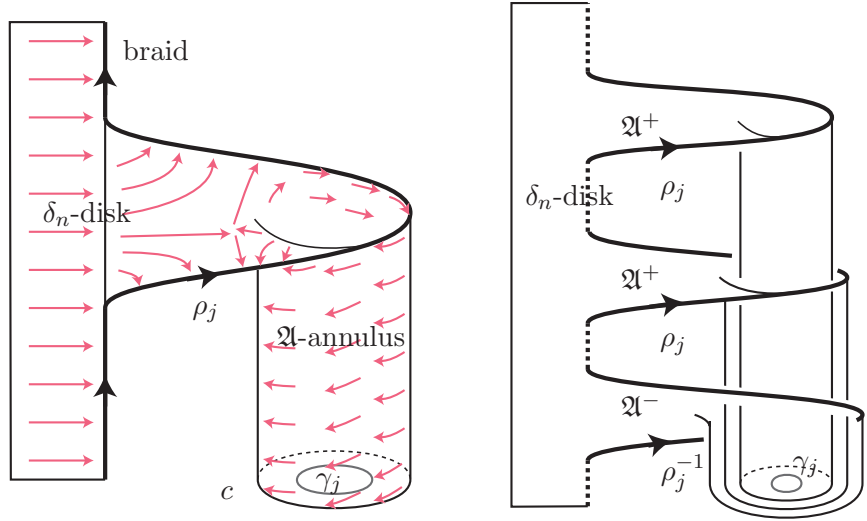}
\end{center}
\caption{(Left) the characteristic foliation on a positive $\mathfrak A$-annulus in $S \times [0,1]$ has a positive hyperbolic point.
(Right) three $\mathfrak A$-annuli near binding $\gamma_j$.}
\label{A-annulus}
\end{figure}
An $\mathfrak A$-annulus associated to $\rho_j$ (resp. $\rho_j^{-1}$) is called {\em positive} (resp. {\em negative}) and denoted by $\A^+$ (resp. $\A^-$).
The boundary consists of 
\begin{description}
\item[(A)] a part of the braid representing $\rho_j$ (resp. $\rho_j^{-1}$), 
\item[(B)] a part of the boundary of disk $\delta_n$ where the $\mathfrak A$-annulus is glued to, and
\item[(C)] a circle $c$ in the page $S\times \{0\}$ once around $\gamma_j$.
\end{description}
\end{definition}
As sketched in Figure~\ref{A-annulus}, an $\mathfrak A^+$-annulus  (resp. $\A^-$) has one positive (resp. negative) hyperbolic singularity in the characteristic foliation.
Along part (A, C) (resp. (B)) of the boundary, the characteristic foliation is outward (resp. inward).  
Around each binding $\gamma_i$ ($i=2, \cdots, r$), the $\mathfrak A$-annuli are standing up disjointly and parpendicular to $S\times \{0\}$.

The orientation of the braid $b$ induces the orientation of each $\mathfrak A$-annulus and its $c$-circle. 
If a $c$-circle is the boundary of a positive $\A$-annulus (i.e., clockwise oriented), then we call it a {\em positive} $c$-circle. 
With this orientation, the algebraic count of the concentric $c$-circles around $\gamma_j$ is equal to $a_{\rho_j}$.

We number the $c$-circles $c_1, c_2, \cdots,$ around $\gamma_j$ ($j=2, \cdots, r$) from the innermost one. 
Suppose that $c_1, c_2, \cdots, c_k$ have the same sign but $c_k$ and $c_{k+1}$ have opposite signs. 
We connect the two $\mathfrak A$-annuli associated to $c_k$ and $c_{k+1}$ by identifying $c_k, c_{k+1}$ then round the corner. 
This operation does not create new singularities in the characteristic foliation.
The orientations of the $\mathfrak A$-annuli match when they are glued.
We push up the bottom of the glued $\A$-annuli so they are contained in $S \times [\epsilon, 1]$.

Renumber the remaining $c$-circles $c_1, c_2, \cdots,$ to which we apply the same procedure above.
After finitely many times, all the $c$-circles around $\gamma_j$ have the same sign, and there are $|a_{\rho_j}|$ of them. 
See Figure~\ref{cancelation}.
%
\begin{figure}[htpb!]
\begin{center}

\psfrag{1}{$c_1$}
\psfrag{2}{$c_2$}
\psfrag{3}{$c_3$}
\psfrag{4}{$c_4$}
\psfrag{5}{$c_5$}
\psfrag{g}{$\gamma_j$}
\psfrag{i}{identify}
\psfrag{c}{$c_2$ and $c_3$}

\includegraphics[width=1\textwidth]{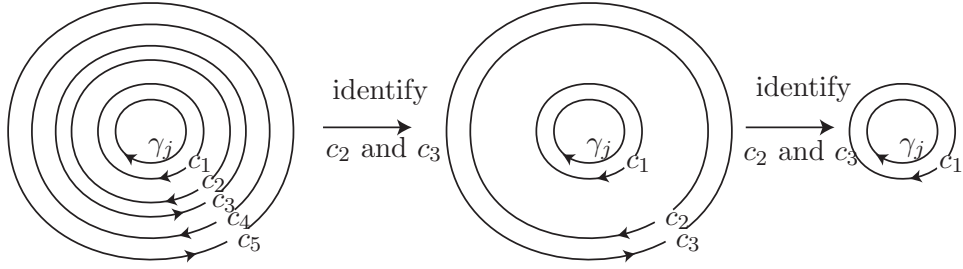}
\end{center}
\caption{Construction of $\tilde \Sigma_b$.}
\label{cancelation}
\end{figure}
We have obtained a smooth surface embedded in $M_{(S, \phi)}$, which we denote by $\tilde \Sigma_b$. 
Its oriented boundary consists of the braid $b$ and $|a_{\rho_j}|$ copies of $c$-circles of ${\rm sgn}(a_{\rho_j})$ around $\gamma_j$ for each $j=2, \cdots, r$.

\section{Construction of Seifert surface $\Sigma$}\label{construction-section}

In this section, we construct surfaces $\Sigma_0, \Sigma_1, \cdots, \Sigma_l \subset S\times [0,1]$ inductively from $s_i$ copies of rectangle $\mathcal D_i$ (defined shortly) and $\tilde\Sigma_b$ constructed in Section~\ref{A-annulus-section}. 
About the integer $s_i$ we assume:

\begin{proposition}\label{positivity of s} 
We may assume that $s_i\geq 0$ for $i=2, \cdots, r$.
\end{proposition}

The purpose of the assumption is that we want the surface $\Sigma_l$ satisfying 
$$
(s_i \mbox{ copies of } \phi(d_i)) \subset \Sigma_l \cap (S\times \{0\})\mbox{ and }
(s_i \mbox{ copies of } d_i) \subset \Sigma_l \cap (S\times \{1\})
$$
for each $i=2, \cdots, r,$ so that the arcs can be identified under the monodromy $\phi$, which in Section~\ref{glueing-section} yields our desired Seifert surface $\Sigma$ for $b$.

\begin{proof}

Recall that a positive braid-stabilization (Definition~\ref{def of stabilization}) about any binding component preserves the transverse isotopy class of the braid.

As in Figure~\ref{stabilization}, put a point $x_{n+1}$ between $\gamma_1$ and $\alpha_1$ on the left side of $x_n$ and put a point $\nu$ between $\gamma_i$ and $\alpha_i$.
\begin{figure}[htpb!]
\begin{center}
\psfrag{G}{$\gamma$}
\psfrag{d}{$d$}
\psfrag{N}{$\nu$}
\psfrag{x}{$x$}
\psfrag{i}{$\scriptstyle i$}
\psfrag{m}{$x_{n+1}$}
\psfrag{a}{$\alpha$}
\psfrag{1}{$\scriptstyle1$}
\psfrag{2}{$\scriptstyle 2$}
\psfrag{3}{$\scriptstyle 3$}
\psfrag{4}{$\scriptstyle 4$}
\psfrag{r}{$\scriptstyle r$}
\psfrag{k}{$\kappa$}
\psfrag{l}{$\kappa'$}
\psfrag{n}{$x_n$}
\includegraphics[width=.6\textwidth]{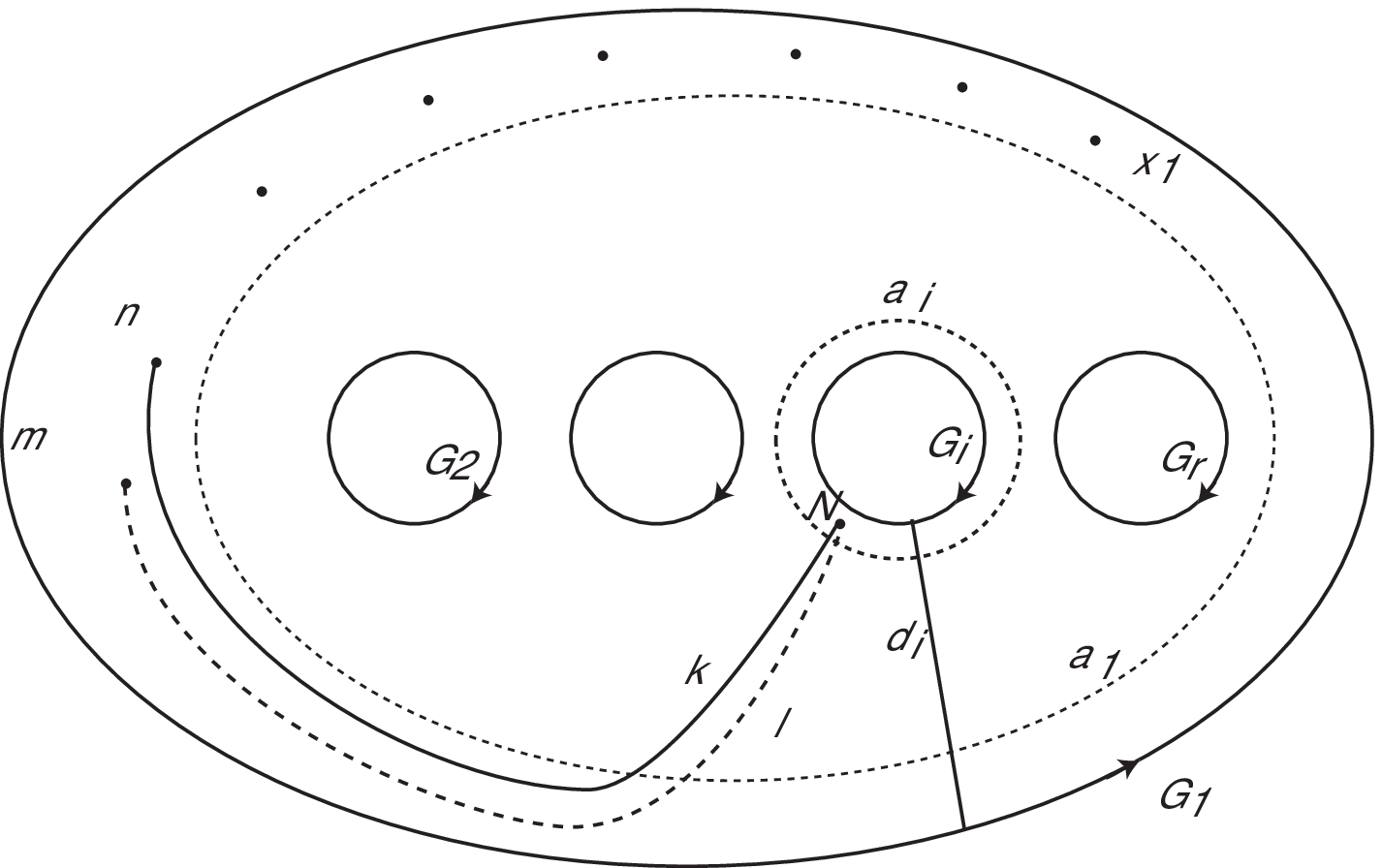}
\end{center}
\caption{}
\label{stabilization}
\end{figure}
We apply a positive braid stabilization about $\gamma_i$ to a small line segment of the $n$-th strand of $b$ along the arc $\kappa$ joining $x_n$ and $\nu$ as depicted in Figure~\ref{stabilization}.
Now $b$ is $(n+1)$-stranded, and the $(n+1)$-th strand intersects the page $S \times \{1\}$ in the point $\nu$.
Note that the stabilized braid does not satisfy Assumption~\ref{assumption} because $\nu$ is away from $\gamma_1$. 
Using braid isotopy supported in $S \times [ - \epsilon, \epsilon]$ we drag $\nu$ back to $x_{n+1}$ along the dashed arc $\kappa'$ as in Figure~\ref{stabilization} so that Assumption~\ref{assumption} is satisfied. 
As a consequence, the dragging operation replaces the braid segment $\{\nu\}\times (1-\epsilon, 1)$ with a copy of $\kappa'$ embedded in $S \times (1-\epsilon, 1)$, and it also replaces $\{\nu\}\times (0,\epsilon)$ with a copy of $\phi(\kappa')$ in $S\times (0, \epsilon)$.
We denote the resulting braid by $b^+$.
In the homology group $H_1(S\times [-\epsilon,\epsilon];\Z)=H_1(S;\Z)$, we have 
\begin{equation}\label{difference}
[b^+ \setminus b] =[d_i]-\phi_*[d_i].
\end{equation}
Hence by (\ref{rewrite}) and (\ref{def of s_j}), the positive stabilization along $\kappa$ imposes an increase of $s_i$ by $1$, 
whereas other $s_j$ ($j\neq i$) remain the same.

Moreover, it turns out that the equality (\ref{difference}) holds regardless of the choice of stabilization arc. 
Suppose that $\iota$ (resp. $\iota'$) is another arc joining $x_n$ (resp. $x_{n+1}$) and $\nu$.
Then 
$$
[\kappa \cup \phi(-\kappa)]- [\iota \cup \phi(-\iota)] =
[\kappa - \iota] - [\phi(\kappa - \iota)] = 0 
\quad \mbox{ in } H_1(S;\Z),
$$
since $\kappa-\iota$ is a closed curve and $\phi_*[\gamma_j] = [\gamma_j]$ for any $j=2, \cdots, r$.

Overall, after sufficiently many positive stabilizations about $\gamma_i$, $s_i$ becomes non-negative.
\end{proof}

\begin{definition}\label{def-rectangle-D}
(\textbf{Rectangle $\mathcal D$})
For $j=2, \cdots, r$, we denote the rectangle $d_j \times [0,1]$ in $S\times [0,1]$ by $\mathcal D_j$.
We orient $\mathcal D_j$ so that $\gamma_j$ and $\mathcal D_j$ intersect positively.
\end{definition}

Recall that $\phi=  {(A_{i_l, j_l})}^{\epsilon_l} \cdots (A_{i_1, j_1})^{\epsilon_1} (A_r)^{k_r} \cdots (A_2)^{k_2}$ where
$\epsilon_i \in \Z \setminus\{0\},\  k_i  \in \Z.$ 
Let $\phi_0 =  (A_r)^{k_r} \cdots (A_2)^{k_2}.$ For $m=1, \cdots, l,$  define $\phi_m = {(A_{i_m, j_m})}^{\epsilon_m} \phi_{m-1}$ inductively.
In particular, $\phi_l=\phi$.

\subsection{Surface $\Sigma_0$}\label{Sigma0-section}

\qquad \newline
We cut open the surface $\tilde \Sigma_b \subset M_{(S, \phi)}$ by the page $S\times \{1\}$, and call it $\tilde\Sigma_b \subset S\times [0,1]$ using the same notation.

If $k_j a_{\rho_j}=0$ for all $j=2, \cdots, r$, then define an immersed surface $\Sigma_0 \subset S \times [0,1]$ by
$\Sigma_0 := \tilde \Sigma_b \cup (s_2 \mbox{ copies of } \mathcal D_2) \cup \cdots \cup (s_r \mbox{ copies of } \mathcal D_r).$
Recall that Proposition~\ref{positivity of s} guarantees $s_j\geq0$.

Suppose $k_j a_{\rho_j} > 0$ for some $j$.
Each $\mathcal D_j$ intersects every $\mathfrak A^{\pm}$-annulus around $\gamma_j$ transversely in a simple curve, whose one end sits on the $\rho_j^{\pm}$-curve, a part of the braid $b$, and the other end sits on the page $S\times \{0\}$. 
See the left sketches in Figure~\ref{resolution0}.
%
\begin{figure}[htpb!]
\begin{center}
\psfrag{-}{$\scriptstyle -$}
\psfrag{+}{$\scriptstyle +$}
\psfrag{g}{$\gamma_j$}
\psfrag{D}{$\mathcal D_j$}
\psfrag{A}{$\mathfrak A^+$}
\psfrag{a}{$\mathfrak A^-$}
\psfrag{b}{$\rho_j$}
\psfrag{B}{$\rho_j^{-1}$}

\includegraphics[width=.8\textwidth]{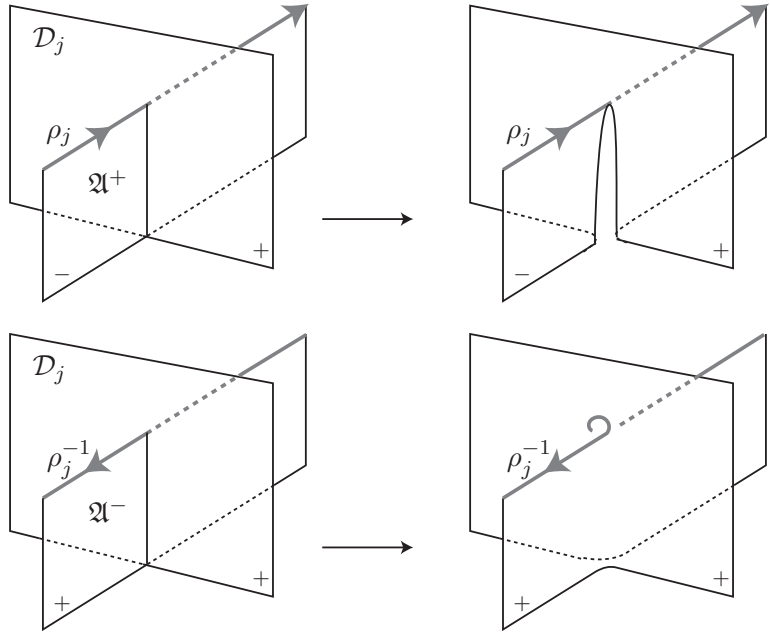}
\end{center}
\caption{We cut along the intersection between a $\mathcal D_i$-rectangle and $\mathfrak A$-annuli and re-glue so that the orientations match.  }
\label{resolution0}
\end{figure}
Consider the intersection curves between $s_j$ copies of $\mathcal D_j$ and the $s_j |k_j|$ innermost $\mathfrak A$-annuli around $\gamma_j$.
Cut them open along the intersections and re-glue so that the orientations of $\mathcal D_j$s and $\mathfrak A$-annuli match. 
See Figure~\ref{resolution0}.
As a consequence, since the signs of $k_j$, $a_{\rho_j}$ and the $\A$-annuli are the same, each arc $d_j= \mathcal D_j \cap (S \times \{0\})$ is replaced with $\phi_0 (d_j) = (A_j)^{k_j} (d_j)$.
See Figure~\ref{resolution1}.
%
\begin{figure}[htpb!]
\begin{center}

\psfrag{+}{$\scriptstyle +$}
\psfrag{g}{$\gamma_j$}
\psfrag{D}{$d_j$}
\psfrag{d}{$A_j^{k_j} d_j$}

\includegraphics[width=.9\textwidth]{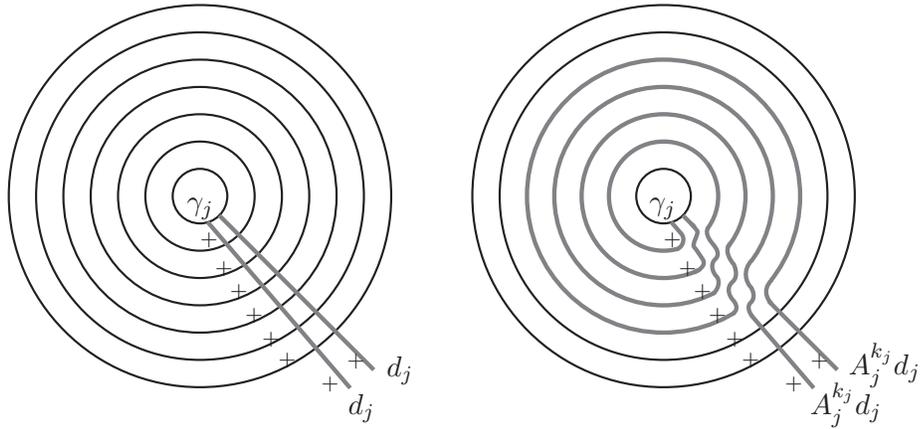}
\end{center}
\caption{After cut and re-glue operation, the arc $d_j$ is replaced by $(A_j)^{k_j} d_j$, where $a_{\rho_j}=6, s_j=2, k_j=2.$ The positive sides of  the surfaces are indicated by ``$+$''.}
\label{resolution1}
\end{figure}

Suppose $k_j a_{\rho_j} < 0$ for some $j$.
By braid isotopy, we deform braid from $b$ to 
$$b (\rho_j)^{-k_j s_j} (\rho_j)^{k_j s_j} \quad \mbox{(read from the left to the right)}$$
which introduces new $2 |k_j| s_j$ innermost {\em dummy} $\mathfrak A$-annuli  around $\gamma_j$. 
Apply the same cut and re-glue operations as in the above paragraph to $s_j$ copies of $\mathcal D_j$ and the $s_j |k_j|$ innermost $\A$-annuli of ${\rm sgn}(k_j)$ so that their orientations match when glueing.
Because of the dummy $\mathfrak A$-annuli, the curve $d_j= \mathcal D_j \cap (S \times \{0\})$ is replaced with $\phi_0 (d_j)= (A_j)^{k_j} (d_j)$.

In either case, starting from $s_j$ copies of $\mathcal D_j$ and $\tilde\Sigma_b$, we have obtained an immersed surface, $\Sigma_0$, in $S\times [0,1]$.
The boundary of $\Sigma_0$ consists of 
\begin{eqnarray*}
\partial \Sigma_0 &=& 
b \cup ((\delta_1 \cup \cdots \cup \delta_n) \cap (S\times \{0\})) 
\cup 
((\delta_1 \cup \cdots \cup \delta_n) \cap (S\times \{1\}))\\
&& 
\cup \ 
((d_2 \cup \cdots \cup d_r) \times \{1\}) 
\ \cup \ 
(\phi_0 (d_2 \cup \cdots \cup d_r) \times \{0\}) 
\\
&& \cup \ (c\mbox{-circles around } \gamma_2, \cdots, \gamma_r),
\end{eqnarray*}
here $d_j \times \{1\}$ and $\phi_0(d_j)\times \{0\}$ mean the $s_j$ copies of each.

For later purpose, we investigate the characteristic foliations under the cut and re-glue operations. 
See Figure~\ref{resolution-2}.
%
\begin{figure}[htpb!]
\begin{center}
\begin{picture}(370, 144)
\put(0,0){\includegraphics{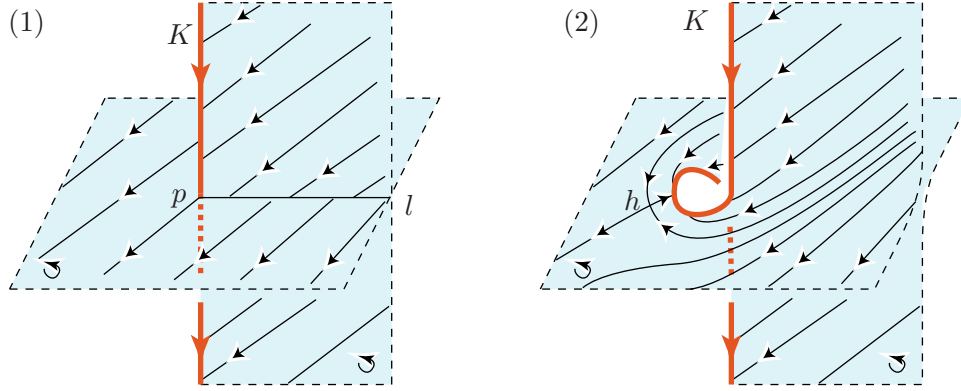}}
\put(0, 135){(1)}
\put(210, 135){(2)}
\put(150, 65){$l$}
\put(62, 70){$p$}
\put(60, 130){$K$}
\put(233, 67){$h$}
\put(255, 135){$K$}
\end{picture}
\caption{(1) A negative intersection $p$. (2) Creation of a negative hyperbolic singularity $h$ by resolving the singular arc $l$.}\label{resolution-2}
\end{center}
\end{figure}
It is a local picture of self-intersection $l$ of an oriented surface and its  characteristic foliation.
The boundary $K$ is a transverse knot intersecting the surface at point $p$ negatively. 
On $l$, we assume that there is no singularity in the characteristic foliation. 
Our cut and re-glue operation creates a hyperbolic singularity.
The signs of the new hyperbolic points are determined in the following way:

\begin{proposition}\label{prop-sign}
\cite[Proposition 3.8]{KP}
If $p$ is a positive $($negative$)$ transverse intersection of $K$ and the surface, then the new hyperbolic point has positive $($negative$)$ sign.
\end{proposition}

\subsection{Surfaces $\Sigma_1, \cdots, \Sigma_l$}\label{Sigmal-section}

\qquad \newline
Suppose that we have constructed immersed surfaces $\Sigma_0, \cdots, \Sigma_{k-1} \subset S \times [0,1]$ satisfying: 
\begin{equation}
\label{boundary-equation}
\Sigma_m \cap (S \times \{0\}) =
((\delta_1 \cup \cdots \cup \delta_n) \cap (S\times \{0\})) 
\ \cup \
\phi_m (d_2 \cup \cdots \cup d_r) \ 
\cup \ (c\mbox{-circles}).
\end{equation}
Here $\phi_m(d_j)$ means its $s_j$ copies, because we started the construction of $\Sigma_0$ from $s_j$ copies of $\mathcal D_j$.
We deform $\Sigma_{k-1}$ to obtain $\Sigma_k$ by applying the following two kinds of surgery.

\noindent
(\textbf{Surgery 1})
Recall that $\phi_k = {(A_{i_k, j_k})}^{\epsilon_k} \phi_{k-1}$.
Suppose $a \subset \phi_{k-1}(d_{i_k})$ and $a' \subset \phi_{k-1}(d_{j_k})$ are sub-arcs which have geometric intersection 
$$
i(a, \alpha_{i_k, j_k})=i(a', \alpha_{i_k, j_k})=1
$$ 
and end at $\gamma_{i_k}$ and $\gamma_{j_k}$ respectively.
See Figure~\ref{T-tunnel}-(1).
Note that there exist $s_{i_k}$ ($s_{j_k}$) parallel copies of $a$ ($a'$) satisfying such conditions. 
%
\begin{figure}[htpb!]
\begin{center}

\psfrag{1}{(1)}
\psfrag{2}{(2)}
\psfrag{3}{(3)}
\psfrag{4}{(4)}
\psfrag{5}{(5)}
\psfrag{i}{$\gamma_{i_k}$}
\psfrag{j}{$\gamma_{j_k}$}
\psfrag{A}{$\alpha_{i_k, j_k}$}
\psfrag{a}{$a$}
\psfrag{b}{$a'$}
\psfrag{C}{$c$-circles}
\psfrag{S}{$\Sigma_{k-1}$}
\psfrag{+}{$\scriptstyle +$}
\psfrag{T}{$T$-tunnels}
\psfrag{t}{$T$-tunnel}
\psfrag{U}{$\A$}
\psfrag{N}{nested $T$-tunnels}
\includegraphics[width=1\textwidth]{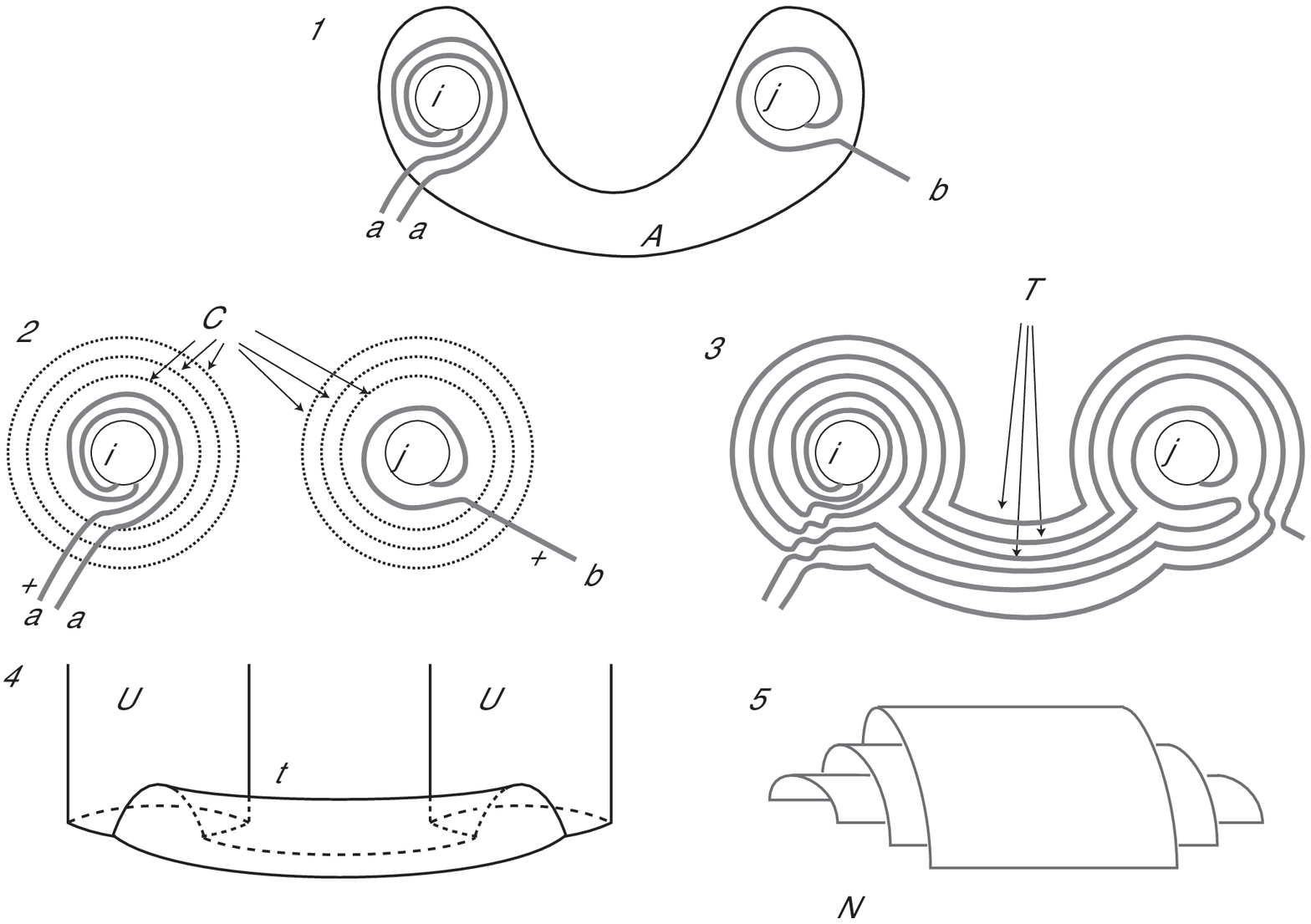}
\end{center}
\caption{
[Surgery 1] 
(1) 
Arcs $a, a'$ have geometric intersection number $1$ with $\alpha_{i_k, j_k}$, where $s_{i_k}=2, s_{j_k}=1$. \\
(2)
On the page $S\times \{0\}$ we reserve $(s_{i_k}+ s_{j_k}) |\epsilon_k|$ many $c$-circles around each $\gamma_{i_k}$ and $\gamma_{j_k}$, where $\epsilon_k=1$.
The sign of the $c$-circles is $\rm{sgn}(\epsilon_k)$.\\
(3) 
Join $\A$-annuli by disjointly nested $(s_{i_k}+ s_{j_k}) |\epsilon_k|$ many $T$-tunnels. 
The arcs $a, a'$ are replaced by $(A_{i_k, j_k})^{\epsilon_k}a$ and $(A_{i_k, j_k})^{\epsilon_k}a'$.}
\label{T-tunnel}
\end{figure}

By inserting dummy $\A$-annuli around $\gamma_{i_k}$ and/or $\gamma_{j_k}$, if necessary, as we did in Section~\ref{Sigma0-section}, we reserve the innermost $(s_{i_k}+ s_{j_k}) |\epsilon_k|$ many $\A$-annuli around each $\gamma_{i_k}$ and $\gamma_{j_k}$ (Figure~\ref{T-tunnel}-(2)), so that the sign of $\epsilon_k$ and the signs of these $\A$-annuli coincide. 
We cut open the intersections of $\Sigma_{k-1}$ and these $\A$-annuli and re-glue as in Figure~\ref{resolution0}.
There are $(s_{i_k}+ s_{j_k})^2 |\epsilon_k|$ intersection curves.
Further, we connect $(s_{i_k}+ s_{j_k}) |\epsilon_k|$ many $\A$-annuli around $\gamma_{i_k}$ and $(s_{i_k}+ s_{j_k}) |\epsilon_k|$ many $\A$-annuli around $\gamma_{j_k}$ by disjointly nested tunnels from the outermost pairs.
We call such tunnels {\em $T$-tunnels}.
See the passage (2)$\Rightarrow$(3) and (4, 5) in Figure~\ref{T-tunnel}.
Round the corners where the $T$-tunnels and the $\A$-annuli meeting. 
Orientation of the $T$-tunnel is induced by the orientation of the $\A$-annuli.
We observe that the arc $a$ (resp. $a'$) is replaced by $(A_{i_k, j_k})^{\epsilon_k}a$ (resp. $(A_{i_k, j_k})^{\epsilon_k}a'$).

\begin{proposition}\label{T-tunnel-prop}
After perturbation, a $T$-tunnel has one hyperbolic singularity in the characteristic foliation. The sign of the hyperbolic point is $- {\rm sgn}(\epsilon_k)$.
\end{proposition}

\begin{proof}
See Figure~\ref{T-tunnel-foliation}.
Let $p_1, \cdots, p_4$ (resp. $p_5, p_6$) be points on a $c$-circle around $\gamma_{i_k}$ (resp. $\gamma_{j_k}$) which is engaged in Surgery 2.
%
\begin{figure}[htpb!]
\begin{center}

\psfrag{1}{$\scriptstyle p_1$}
\psfrag{2}{$\scriptstyle p_2$}
\psfrag{3}{$\scriptstyle p_3$}
\psfrag{4}{$\scriptstyle p_4$}
\psfrag{5}{$\scriptstyle p_5$}
\psfrag{6}{$\scriptstyle p_6$}
\psfrag{i}{$\gamma_{i_k}$}
\psfrag{j}{$\gamma_{j_k}$}
\psfrag{A}{$\A$-annulus around $\gamma_{i_k}$}
\psfrag{a}{$a$}
\psfrag{b}{$a'$}
\psfrag{C}{$c$-circles}
\psfrag{S}{$\Sigma_{k-1}$}
\psfrag{+}{$\scriptstyle +$}
\psfrag{T}{$T$-tunnel}

\includegraphics[width=.8\textwidth]{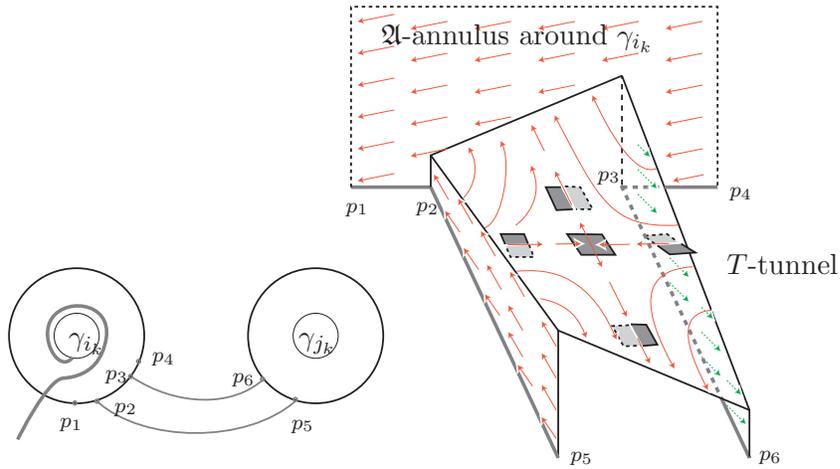}
\end{center}
\caption{The characteristic foliation and contact planes on a $T$-tunnel (the corners should be rounded) has a hyperbolic singularity. 
}
\label{T-tunnel-foliation}
\end{figure}
As sketched, we perturb the $T$-tunnel so that the heights above points $p_2$ and $p_6$ are shorter than the heights above $p_3$ and $p_5$.
By Assumption~\ref{foliation-assumption} it imposes the $T$-tunnel one hyperbolic point.
The signs of both of the $\A$-annuli are equal to the sign of $\epsilon_k$.
If ${\rm sgn}(\epsilon_k) = +$, then the negative side the $T$-tunnel is facing to us.
Hence the sign of the hyperbolic point is negative.
\end{proof}

\noindent
(\textbf{Surgery 2})
Let $a \subset \phi_{k-1} (d_2 \cup \cdots \cup d_r) \subset S \times \{0\}$ be a sub-arc that has the geometric intersection number 
$i(a, \alpha_{i_k, j_k})=2.$
See Figure~\ref{tunnel}-(1).
%
\begin{figure}[htpb!]
\begin{center}

\psfrag{i}{$\gamma_{i_k}$}
\psfrag{j}{$\gamma_{j_k}$}
\psfrag{A}{$\alpha_{i_k, j_k}$}
\psfrag{a}{$a$}
\psfrag{S}{$\Sigma_{k-1}$}
\psfrag{1}{(1)}
\psfrag{2}{(2)}
\includegraphics[width=1\textwidth]{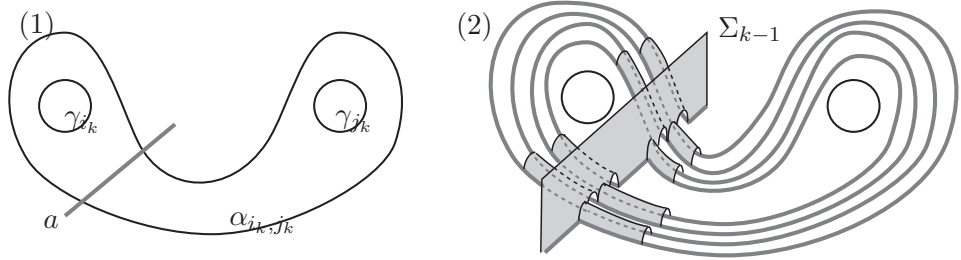}
\end{center}
\caption{[Surgery 2] 
(1) Sub-arc $a$ in $S\times \{0\}$. 
(2) Part of the $U$-tunnels for $a$ where $\epsilon_k=2$.}
\label{tunnel}
\end{figure}
To each side of the arc $a$ we add $|\epsilon_k|$ many tunnels, called $U$-{\em tunnels}, embedded in $S \times [0, \epsilon]$ ($\epsilon \ll 1$). Their feet do not touch the feet of $T$-tunnels of Surgery 1, and are contained in the arc $(A_{i_k, j_k})^{\epsilon_k}a$. 
See Figure~\ref{tunnel}-(2). 
Round the corners where the $U$-tunnels and the original surface $\Sigma_{k-1}$ meet. 
The orientations of the $U$-tunnels are induced by that of $\Sigma_{k-1}$. 
If arc $a$ is already a part of the feet of a $U$-tunnel in $\Sigma_{k-1}$, then the new tunnel is dug lower than the existing tunnel.

In general, since we have multiple of sub-arcs with $i(a, \alpha_{i_k, j_k})=2$ we construct nested mutually disjoint $U$-tunnels. 
See Figure~\ref{nested-tunnels}.
%
\begin{figure}[htpb!]
\begin{center}

\psfrag{i}{$\gamma_{i_k}$}
\psfrag{j}{$\gamma_{j_k}$}
\psfrag{A}{$\alpha_{i_k, j_k}$}
\psfrag{a}{arc $a$}
\psfrag{b}{arc $a'$}
\psfrag{C}{$(A_{i_k, j_k})^{\epsilon_j}a$}
\psfrag{d}{$(A_{i_k, j_k})^{\epsilon_j}a'$}
\psfrag{1}{(1)}
\psfrag{2}{(2)}

\includegraphics[width=1\textwidth]{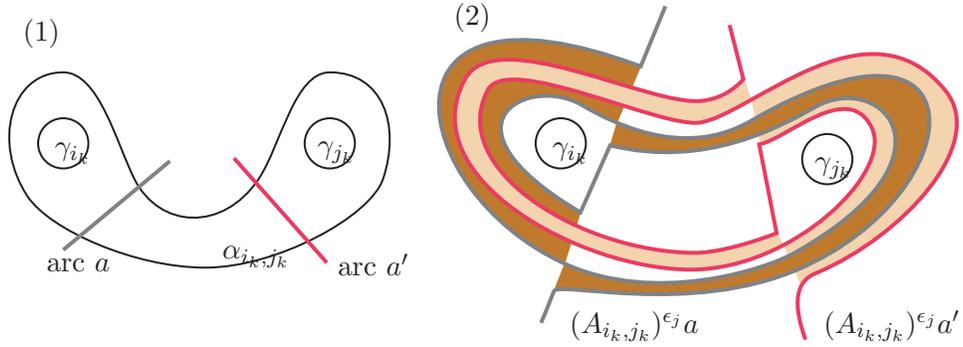}
\end{center}
\caption{Arcs $(A_{i_k, j_k})^{\epsilon_k}a$ and $(A_{i_k, j_k})^{\epsilon_k}a'$, where $\epsilon_k = 1$, form nested disjoint $U$-tunnels, dark shaded and lightly shaded.}
\label{nested-tunnels}
\end{figure}

\begin{proposition}\label{count of U-tunnel}
Up to small perturbation, each $U$-tunnel has one hyperbolic singularity in the characteristic foliation. 
For each arc $a$ with $i(a, \alpha_{i_k, j_k})=2$, the signs of the hyperbolic points on two $U$-tunnels are identical if and only if the $U$-tunnels are on the same side of $a$. 
\end{proposition}

\begin{proof}
By Assumption~\ref{foliation-assumption}-(2), the characteristic foliation near $\alpha_{i_k, j_k}$ is as sketched in Figure~\ref{foliation-of-tunnel}-(1).
%
\begin{figure}[htpb!]
\begin{center}

\psfrag{i}{$\gamma_{i_k}$}
\psfrag{j}{$\gamma_{j_k}$}
\psfrag{A}{$\alpha_{i_k, j_k}$}
\psfrag{1}{$p_1$}
\psfrag{2}{$p_2$}
\psfrag{3}{$p_3$}
\psfrag{4}{$p_4$}
\psfrag{a}{(1)}
\psfrag{b}{(2)}
\psfrag{c}{(3)}

\includegraphics[width=0.7\textwidth]{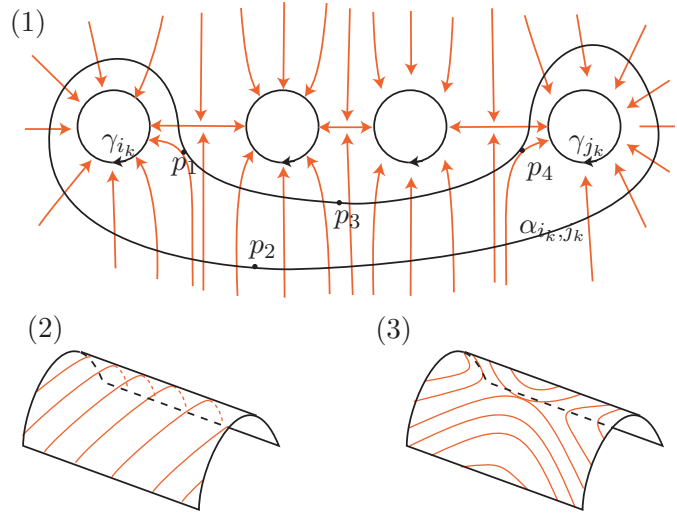}
\end{center}
\caption{(1) The characteristic foliation on $S\times \{0\}$ near $\alpha_{i_k, j_k}$.
(2) $U$-tunnel near $p_2$ and $p_3$.
(3) $U$-tunnel near $p_1$ and $p_4$. }
\label{foliation-of-tunnel}
\end{figure}
Near the points where the characteristic foliation transversely intersects $\alpha_{i_k, j_k}$ (for example, points $p_2$ and $p_3$), 
there are no singularities in the foliation of the U-tunnels as depicted in Figure~\ref{foliation-of-tunnel}-(2).

Let $p_1, p_4$ be the points where the leaves are tangent to $\alpha_{i_k, j_k}$.
We may assume that every arc $a$ with $i(a, \alpha_{i_k, j_k})=2$ separates $p_1$ and $p_4$. 
Each $U$-tunnel contains a single hyperbolic point near $p_1$ or $p_4$ as sketched in Figure~\ref{foliation-of-tunnel}-(3).
Note that the positive normals to two $U$-tunnels are pointing the same direction (outward or inward) if and only if the $U$-tunnels locate in the same side of $a$. 
Hence the signs of two hyperbolic points are the same if and only if the associated $U$-tunnels are located in the same side of $a$.
\end{proof}

The $T$-tunnels and $U$-tunnels intersect in general.
By small perturbation, if necessary, we may assume the intersection curves do not contain any singularities of the characteristic foliation. 
We resolve the intersections by the cut and re-glue operations so that the orientations match. 
Since the intersections are mutually disjoint simple closed curves, no new singular points would be created.

Now we have obtained an immersed surface $\Sigma_k$ in $S \times [0,1]$ that satisfies the boundary condition (\ref{boundary-equation}).
Inductively, we construct immersed surfaces $\Sigma_0, \Sigma_1, \cdots, \Sigma_l$ in $S \times [0,1]$.

We further deform $\Sigma_l$ to obtain an embedded surface: 
Since we have added numerous dummy $\A$-annuli during Surgery 2, the boundary $\partial \Sigma_l$ may contain $c$-circles.
Following the algorithm described near Figure~\ref{cancelation}, we pair up two $\A$-annuli of opposite signs and glue them along their $c$-circles.
Since the algebraic count of such $c$-circle around each $\gamma_j$ ($j=2, \cdots, r$) is $0$, we can remove all the $c$-circles of $\Sigma_l$.
Next, suppose $\A_\bullet$ and $\A_\circ$ are a pair of $\A$-annuli glued  along their $c$-circles. 
The intersections between $\A_\bullet \cup \A_\circ$ and $\Sigma_l$ (or former $\mathcal D$-rectangles) are arcs whose end points are on the braid $b$.
We resolve each intersection by the cut and re-glue operation that creates two new hyperbolic points in the characteristic foliation. 
Since $\A_\bullet$ and $\A_\circ$ have opposite signs, Proposition~\ref{prop-sign} implies that the signs of the hyperbolic points are opposite.

Finally the surface $\Sigma_l$ is embedded in $S \times [0,1]$, whose 
boundary is 
\begin{eqnarray*}
\partial \Sigma_l &=& 
b \cup ((\delta_1 \cup \cdots \cup \delta_n) \cap (S\times \{0\})) \cup ((\delta_1 \cup \cdots \cup \delta_n) \cap (S\times \{1\}))\\
&& 
((d_2 \cup \cdots \cup d_r) \times \{1\}) 
\ \cup \ 
(\phi (d_2 \cup \cdots \cup d_r) \times \{0\}), 
\end{eqnarray*}
here $d_j \times \{1\}$ and $\phi(d_j)\times \{0\}$ mean the $s_j$ copies of each.

\subsection{Glueing $\Sigma_l$ by monodromy to obtain $\Sigma$}\label{glueing-section}
\quad \newline
In order to obtain a desired Seifert surface $\Sigma \subset M_{(S, \phi)}$ for the braid $b$, topologically it is enough to glue the boundary components of $\Sigma_l$;
\begin{eqnarray*}
\delta_i \cap (S\times \{0\})
&\mbox{ with }&
\delta_i \cap (S\times \{1\})
\quad \mbox{ for } i=1, \cdots, n,
\\
s_j\mbox{ copies of }
\phi (d_j) \times \{0\}
&\mbox{ with }&
s_j\mbox{ copies of }
d_j \times \{1\} \quad \mbox{ for } j=2, \cdots, r.
\end{eqnarray*}
Note that near the bindings, the monodromy $\phi$ is the identity map, hence   $\delta_i \cap (S\times \{1\})$ and $\delta_i \cap (S\times \{0\})$ can be identified under $\phi$.

The remaining task is to justify that the characteristic foliations near $\phi (d_j) \times \{0\}$ and $d_j \times \{1\}$ smoothly match. 
To this end, we fix a small $\varepsilon>0$ and  by using braid isotopy we assume that $\Sigma_l$  is ``perpendicular'' to the pages $S \times \{\tau\}$ ($0\leq \tau \leq \varepsilon)$, i.e., letting $t$ be the coordinate for $[0,1]$, at any point $p \in (S \times [0, \varepsilon]) \cap \Sigma_l$ away from the bindings, the vector $(\frac{\partial}{\partial t})_p$ is contained in the tangent plane $T_p\Sigma_l$.
Let $\alpha_0=\beta+ C dt$ denote the contact $1$-form on the page $S \times \{0\}$ away from the bindings, where $\beta$ is a $1$-form on $S$ and  $C \gg 1$. 
By the argument in \cite[p.152]{Ge}, away from the bindings, we may assume that 
\begin{itemize}
\item
the contact form in  $S \times [\varepsilon, 1]$  is constantly 
$\alpha_1=\phi^* \beta + C dt$, 
\item
on the page $S \times \{ s\varepsilon\}$, $s \in [0,1]$, the contact form is 
$$\alpha_s := (1-s) \beta + s (\phi^* \beta) + Cdt.$$
\end{itemize}
Now not only we can glue the boundaries of $\Sigma_l$ by using the monodromy $\phi$, but also the characteristic foliation is smoothly extended. 
Since $\alpha_s(\frac{\partial}{\partial t}) >0,$ if $p$ is a point in $S \times \{s\varepsilon\}$ away from the bindings, the contact plane $\xi_p$ does not contain $(\frac{\partial}{\partial t})_p$.
This shows that there are no singularities in the characteristic foliation on $\Sigma_l \cap (\rm{Int}(S) \times [0, \varepsilon])$.

\section{Proof of Theorem~\ref{sl-formula}}\label{proof-section}

Let $K \subset (M, \xi)$ be a null-homologous transverse knot with a Seifert surface $\Sigma$. 
After small perturbation, we may assume that the characteristic foliation on $\Sigma$ is of Morse-Smale type (see \cite[Definition 4.6.8]{Ge} for definition). 
Let $e^+$ ($e^-$) and $h^+$ ($h^-$) represent the numbers of positive (negative) elliptic and positive (negative) hyperbolic singularities of the characteristic foliation.
It is known (see \cite{E1} for example) that the self linking number of $K$ relative to the homology class $[\Sigma] \in H_2(M, K; \Z)$ satisfies:
\begin{equation}\label{sl-formula-eq}
sl(K, [\Sigma])=-(e^+-e^-)+(h^+-h^-).
\end{equation}

The next lemma investigates the entry $t_{i,j}$ of the monodromy matrix.
\begin{lemma}\label{symmetry of T}
For each $i=2, \cdots, r$, we have
\begin{equation}
\label{eq-of-t}
t_{i,j} = \left\{
\begin{array}{ll}
\displaystyle{\sum_{{1 \leq m \leq l} \atop{(i_m, j_m) = (i, j)}}} \epsilon_m, 
\qquad 
\mbox{if }  j > i, \\
\displaystyle{\sum_{{1 \leq m \leq l} \atop{(i_m, j_m) = (j, i)}}} \epsilon_m, 
\qquad 
\mbox{if } j < i, 
\end{array}
\right.
\end{equation}
\begin{equation}
\label{eq-of-t2}
t_{i,i} = k_i +
\sum_{{1 \leq m \leq l} \atop{i_m=i \ {\rm or}\ j_m=i}} \epsilon_m 
= k_i +  \sum_{{2\leq j \leq r} \atop{j \neq i}} t_{i,j}. 
\end{equation}
In particular $T$ is a symmetric matrix, i.e., $t_{i,j}=t_{j,i}$.
\end{lemma}

\begin{proof}
We orient circles $\alpha_i$ and $\alpha_{i, j}$ counterclockwise. 
In $H_1(M;\Z)$ we have $[\alpha_{i_m, j_m}] = [\alpha_{i_m}] + [\alpha_{j_m}]$ and $[\alpha_i] =- [\gamma_i]$. 
For any $2\leq i<j \leq r$ and $2\leq i' < j' \leq r$ 
$$
[A_{i,j} (\alpha_{i', j'})] = [A_{i,j} (\alpha_{i'})] +  [A_{i,j} (\alpha_{j'})] = [\alpha_{i'}]  + [\alpha_{j'}] = [\alpha_{i', j'}].
$$
Hence by the description of $\phi$ in (\ref{eq of phi}), for each $i=2, \cdots, r$, we have:
\begin{eqnarray*}
[d_i] - \phi_\ast [d_i] 
&=& 
- \sum_{{1 \leq m \leq l} \atop{i_m=i \ {\rm or}\ j_m=i}}
\left[ 
A_{i_l, j_l}^{\epsilon_l} \cdots 
A_{i_{m+1}, j_{m+1}}^{\epsilon_{m+1}} 
(\epsilon_m \alpha_{i_m, j_m})
\right]
- k_i [\alpha_i] \\
&=&
- \sum_{{1 \leq m \leq l} \atop{i_m=i \ {\rm or}\ j_m=i}}
\epsilon_m [\alpha_{i_m, j_m}] - k_i [\alpha_i] \\
&=&
\sum_{{1 \leq m \leq l} \atop{i_m=i \ {\rm or}\ j_m=i}}
\epsilon_m ([\gamma_{i_m}] + [\gamma_{j_m}]) + k_i [\gamma_i]. 
\end{eqnarray*}
Combining with (\ref{rewrite}) it follows that:
$$
\sum_{j=2}^r t_{i,j} [\gamma_j] = 
\sum_{{1 \leq m \leq l} \atop{i_m=i \ {\rm or}\ j_m=i}}
\epsilon_m ([\gamma_{i_m}] + [\gamma_{j_m}]) + k_i [\gamma_i]. $$
Comparing the coefficients, we obtain (\ref{eq-of-t}) and (\ref{eq-of-t2}). 
\end{proof}

Finally we are ready to prove the main theorem.
\begin{proof}[Proof of Theorem~\ref{sl-formula}]

We first investigate the algebraic count of hyperbolic singularities ($h^+ - h^-$) on the Seifert surface $\Sigma$ that we have constructed in Section~\ref{construction-section}.

\begin{itemize}
\item
The twisted bands for the braid word $\sigma_i^{\pm}$ contribute $a_\sigma$. 
\item
Recall an $\A^{\pm}$-annulus has one $\pm$ hyperbolic point, cf Figure~\ref{A-annulus}. 
Hence $\A$-annuli contribute $\Sigma_{j=2}^r a_{\rho_j}$. 
\item
As Figure~\ref{resolution-2} shows, the cut and re-glue operations during the course of constructing $\Sigma_l$ have created new hyperbolic singularities.
The algebraic count of such hyperbolic points is equal to $-\Sigma_{j=2}^r s_j a_{\rho_j}$. 
\item
By Proposition~\ref{T-tunnel-prop}, the $T$-tunnels also contribute to hyperbolic singularities (Figure~\ref{T-tunnel-foliation}). Using (\ref{eq-of-t}), its algebraic count is:
$$
\sum_{j=2}^r s_j 
\sum_{{1\leq m \leq l} \atop{i_m=j \ \rm{or} \ j_m=j}} (- \epsilon_m)
\ = \ 
-\sum_{j=2}^r s_j 
\sum_{{2\leq i \leq r} \atop{i \neq j}} t_{j,i}
$$
\item By Proposition~\ref{count of U-tunnel}, the $U$-tunnels do not contribute to the count.
\end{itemize}
The total algebraic count of hyperbolic points is:
$$
h^+ - h^- 
=
a_\sigma + \sum_{j=2}^r (1-s_j) a_{\rho_j} -\sum_{j=2}^r s_j \sum_{{2\leq i \leq r} \atop{i \neq j}} t_{j,i}
$$
By Definition~\ref{def-rectangle-D}, the intersection of $D_j$ and the binding $\gamma_j$ (resp. $\gamma_1$) turns a positive (resp. negative) elliptic point in the final surface $\Sigma$. 
Since we have used $s_j$ copies of $D_j$ to construct $\Sigma$, 
\begin{eqnarray*}
e^+ 
&=&
(n; \delta{\mbox{-disks}}) + (s_2 +\cdots + s_r), \\
e^- 
&=&
(s_2 +\cdots + s_r).
\end{eqnarray*} 
By (\ref{sl-formula-eq}) the self-linking number is:
$$
sl(b, [\Sigma]) 
= -(e^+ - e^-) +  (h^+ - h^-) 
= -n + a_\sigma + \sum_{j=2}^r a_{\rho_j} (1 - s_j) - \sum_{{2\leq i \leq r} \atop{i \neq j}} t_{j,i}
$$
This completes the proof of Theorem~\ref{sl-formula}. 
\end{proof}

\begin{proof}[Proof of Corollary~\ref{cor-neg-stab}]

A negative stabilization about the binding $\gamma_1$ changes:
\begin{eqnarray*}
n & \mapsto & n+1 
\\
a_\sigma & \mapsto & a_\sigma - 1 
\end{eqnarray*}
which change the quantity of (\ref{sl-formula}) by $-2$.
A negative stabilization about the binding $\gamma_k$, where $k = 2, \cdots, r$, changes:
\begin{eqnarray*}
n & \mapsto & n+1 
\\
a_\sigma & \mapsto & a_\sigma -1 + 2 a_{\rho_k} 
\qquad (\ast)
\\
s_k & \mapsto & s_k + 1 
\quad (\mbox{by the proof of Proposition~\ref{positivity of s}}) 
\\
a_{\rho_j} & \mapsto & a_{\rho_j} + t_{k,j} 
\quad (\mbox{by (\ref{a-rho-and-T})}) 
\end{eqnarray*}
The reason for ($\ast$) is the following.
First, we subtract $1$ from $a_\sigma$ because of the negative kink due to the negative braid stabilization. 
A negative stabilization introduces a new $(n+1)$th strand. 
Let $\rho_k'$ denote the positive winding of the $(n+1)$th strand around $\gamma_k$. 
Let $\sigma_n$ be the usual positive half twist of $n$th and $(n+1)$th strands. 
Then $\rho_k$ and $\rho_k'$ are related to each other by $\rho_k = \sigma_n \rho_k' \sigma_n$. 
This is the reason we need to add $2 a_{\rho_k}$.

Plug the above values into (\ref{sl-formula}) and subtract the original (\ref{sl-formula}), we have 
\begin{eqnarray*}
&& 
\left(
-(n+1) + (a_\sigma -1 + 2 a_{\rho_k})
+ \sum_{j\neq k} (a_{\rho_j} + t_{k,j}) (1-s_j)
\right. \\
&&
\left.
+ (a_{\rho_k} + t_{k,k}) (1-s_k-1) 
- \sum_{j \neq k} s_j \sum_{i \neq j} t_{j,i} 
- (s_k + 1) \sum_{i\neq k} t_{k,i}
\right) \\
&&
-
\left(
-n + a_\sigma + \sum_{j=2}^r a_{\rho_j} (1 - s_j) - \sum_{j=2}^r s_j \sum_{{2\leq i \leq r} \atop{i \neq j}} t_{j,i}
\right) \\
&=&
-2 + 2 a_{\rho_k} + \sum_{j\neq k} t_{k,j} (1-s_j) - a_{\rho_k} 
- t_{k,k} s_k - \sum_{i\neq k} t_{k,i} \\
&=&
-2+ a_{\rho_k} - \sum_{j=2}^r t_{k,j} s_j 
= 
-2 + a_{\rho_k} - \sum_{j=2}^r s_j t_{j,k} 
\stackrel{(\ref{a-rho-and-T})}{=} 
-2,
\end{eqnarray*}
where the second last equation follows by the symmetry of the matrix studied in Proposition~\ref{symmetry of T}.
\end{proof}

A positive braid stabilization induces the same changes in $n, s_k, a_{\rho_j}$ as above, but it changes $a_\sigma$ to $a_\sigma + 1 + 2a_{\rho_k}$. 
A similar calculation shows that our self-linking formula (\ref{sl-formula}) is invariant under a positive braid stabilization. 
Knowing that a positive stabilization preserves the transverse isotopy  class of any braid, this observation justifies our main theorem.

\section*{Acknowledgements}  
The author would like to thank David Gay for helpful conversations, Matt Hedden for Remark~\ref{remark by Matt}, and the referee for thoughtful comments. 
The author was partially supported by NSF grants DMS-0806492.

\bibliographystyle{mrl}\

\begin{thebibliography}{99}

\bibitem{Ben} Bennequin, Daniel. {\em Entrelacements et {\'e}quations de Pfaff,} Ast{\'e}risque, 107-108, (1983) 87-161.


\bibitem{B} Birman, Joan S. {\em Braids, links, and mapping class groups.} Annals of Mathematics Studies, No. 82. Princeton University Press, Princeton, N.J


\bibitem{E1}Etnyre, John B. {\em Planar open book decompositions and contact structures. } Int. Math. Res. Not. 2004, no. 79, 4255-4267. 

\bibitem{EO}Etnyre, John B.; Ozbagci, Burak. {\em Invariants of contact structures from open books.} Trans. Amer. Math. Soc. 360 (2008), no. 6, 3133--3151.

\bibitem{FM}Farb, Benson; Margalit, Dan. {\em A primer on mapping class groups.} Version 5.0

\bibitem{Ge}Geiges, Hansj\"{o}rg. {\em An introduction to contact topology.} Cambridge Studies in Advanced Mathematics, 109. Cambridge University Press, Cambridge, 2008.

\bibitem{G}Giroux, Emmanuel. {\em Contact geometry: from dimension three to higher dimensions.} Proceedings of the International Congress of Mathematicians, Vol. II (Beijing, 2002), 405-414, Higher Ed. Press, Beijing, 2002.


\bibitem{KP} Kawamuro, Keiko; Pavelescu, Elena. {\em The self-linking number in annulus and pants open book decompositions.} 
Algebr. Geom. Topol. 11 (2011) 553-585.

\bibitem{OS} Ozbagci, Burak; Stipsicz, Andr\'as I. {\em Surgery on contact 3-manifolds and Stein surfaces.} Bolyai Society Mathematical Studies, 13. Springer-Verlag, Berlin; J\'anos Bolyai Mathematical Society, Budapest, 2004.


\bibitem{P}Pavelescu, Elena. {\em Braids and Open Book Decompositions.}
Ph.D. thesis, University of Pennsylvania (2008) Available at 
{\tt http://www.math.upenn.edu/grad/dissertations/ ElenaPavelescuThesis.pdf}

\bibitem{TW}Thurston, W. P.; Winkelnkemper, H. E. {\em On the existence of contact forms}. Proc. Amer. Math. Soc. 52 (1975), 345-347.

\end{thebibliography}

\end{document}